\documentclass[12pt]{article}
\usepackage{amsmath,amsfonts,amssymb,amsthm}
\newcommand{\mr}{{\mathbb R}}

\newcommand{\mn}{{\mathbb N}}
\newcommand{\mz}{{\mathbb Z}}
\newcommand{\mc}{{\mathbb C}}
\newcommand{\md}{{\mathbb D}}

\newcommand{\eps}{\varepsilon}

\newcommand{\dist}{\operatorname{dist}}

\newcommand{\Real}{\mathbb{R}}
\newcommand{\Integer}{\mathbb{Z}}
\newcommand{\Complex}{\mathbb{C}}

\newcommand{\todo}[1]{{\sffamily To do:}}
\renewcommand{\Im}{\operatorname{Im}}
\renewcommand{\Re}{\operatorname{Re}}

\newtheorem{thm}{Theorem}
\newtheorem*{thm_nonumber}{Theorem}

\newtheorem{cor}[thm]{Corollary}
\newtheorem{lemma}[thm]{Lemma}

\theoremstyle{remark}{
\newtheorem{rem}{Remark}

}

\title{Inequalities for the eigenvalues of non-selfadjoint Jacobi operators}
\author{Marcel Hansmann\footnotemark[1],
Guy Katriel{\footnotemark[1]\;\;\footnotemark[2]\;\;}}
\date{}

\begin{document}

\maketitle
\renewcommand{\thefootnote}{\fnsymbol{footnote}}
\footnotetext[1]{Institute of Mathematics, Technical University of
Clausthal, 38678 Clausthal-Zellerfeld, Germany.}
\footnotetext[2]{Partially supported by the Humboldt Foundation
(Germany).}
\begin{abstract}
We prove Lieb-Thirring-type bounds on eigenvalues of non-selfadjoint Jacobi
operators, which are nearly as strong as those
proven previously for the case of selfadjoint operators by
Hundertmark and Simon.
We use a method based on determinants of operators and on complex function theory, extending
and sharpening earlier work of Borichev, Golinskii and Kupin.
\end{abstract}

\section{Introduction and Results}

This paper is concerned with the study of the set of discrete eigenvalues of complex Jacobi operators $J: l^2(\mz) \to l^2(\mz)$, represented by a two-sided infinite tridiagonal matrix
$$ J= \left(
      \begin{array}{ccccccc}
       \ddots & \ddots & \ddots &  &  &  &  \\
        &  a_{-1} & b_0 & c_0 &  &  &   \\
        &  & a_0 & b_1 & c_1 &  &   \\
        &  &  & a_1 & b_2 & c_2 &  \\
        &  &  &  & \ddots & \ddots & \ddots
      \end{array}
    \right),
$$
where $\{a_k\}_{k \in \mz}, \{b_k\}_{k \in \mz}$ and $ \{c_k\}_{k \in \mz}$ are bounded complex sequences. More precisely, for $u \in l^2(\mz)$, $J$ is defined via
$$(J u)(k)=a_{k-1}u(k-1)+b_k u(k)+c_k u(k+1).$$
We are interested in operators $J$ that are compact perturbations of the free Jacobi operator $J_0$, which is defined as the special case with
$a_k=c_k \equiv 1$ and $b_k\equiv 0$, i.e.
$$(J_0 u)(k)=u(k-1)+ u(k+1).$$

So, in the following, we will always  assume that $J-J_0$ is compact, or equivalently that
$$\lim_{|k|\rightarrow \infty}a_k=\lim_{|k|\rightarrow
\infty}c_k=1,\quad \text{ and } \quad \lim_{|k|\rightarrow \infty}b_k=0.$$
As is well-known, the spectrum $\sigma(J_0)$ of $J_0$ is equal to
$[-2,2]$ and the compactness of $J-J_0$ implies   that
$\sigma(J) = [-2,2]\; \dot{\cup}\;
\sigma_{d}(J)$, where $[-2,2]$ is the essential spectrum of $J$
and the discrete spectrum $\sigma_{d}(J) \subset \mc \setminus [-2,2]$
consists of a countable set of isolated eigenvalues of finite
algebraic multiplicity, with possible accumulation points in
$[-2,2]$.

We define the sequence $d= \{d_k\}_{k \in \mz} $ as follows
\begin{equation}\label{defdk}d_k=\max\Big(|a_{k-1}-1|,|a_{k}-1|,|b_k|,|c_{k-1}-1|,|c_{k}-1|\Big).\end{equation}
Note that the compactness of $J-J_0$ is equivalent to $\{d_k\}$
converging to $0$. The main results of this paper provide
information on $\sigma_d(J)$, assuming the stronger condition that
$d \in l^p(\mz)$, the space of $p$-summable sequences
(we will see below that $d \in l^p(\mz)$ implies that $J-J_0$ is an element of the Schatten class $\mathbf{S}_p$, see Section \ref{prelim} for
relevant definitions).

\begin{thm}\label{thm2}
Let $\tau \in (0,1)$. If $d \in l^p(\mz)$, where $p > 1$, then
\begin{equation}\label{inequality2}
\sum_{\lambda \in \sigma_d(J)}
\frac{\dist(\lambda,[-2,2])^{p+\tau}}{|\lambda^2-4|^{\frac{1}{2}}} \leq
C(p,\tau)\|d\|_{l^p}^p.
\end{equation}
Furthermore, if $d \in l^1(\mz)$, then
\begin{equation}\label{inequality2b}
\sum_{\lambda \in \sigma_d(J)}
\frac{\dist(\lambda,[-2,2])^{1+\tau}}{|\lambda^2-4|^{\frac{1}{2}+\frac {\tau}{4}}} \leq
C(\tau)\|d\|_{l^1}.
\end{equation}
\end{thm}

\begin{rem}
In the summation above, and elsewhere in this article, each eigenvalue is counted according to its algebraic multiplicity. Furthermore, the constants used in this paper are generic, i.e. the value of a constant may change from line to line. However, we will always indicate the parameters that a constant depends on.
\end{rem}

The above estimates can be regarded as
`near-generalisations' of the following Lieb-Thirring inequalities proven by
Hundertmark and Simon \cite{hundertsimon} for selfadjoint Jacobi operators (i.e. for the
case that $a_k,b_k,c_k \in \mr$ and $a_k=c_k$ for all $k$).

\begin{thm_nonumber}[\cite{hundertsimon}, Theorem 2]
Let $J$ be selfadjoint and suppose that $d \in l^p(\mz)$, where $p\geq 1$. Then
\begin{equation}\label{hs2}
  \sum_{\lambda \in \sigma_d(J),\; \lambda < -2} |\lambda + 2|^{p-\frac 1 2} +  \sum_{\lambda \in \sigma_d(J),\; \lambda>2} |\lambda - 2|^{p- \frac 1 2} \leq C(p)\|d\|_{l^p}^p.
\end{equation}
\end{thm_nonumber}


To see the relation between the above result and Theorem \ref{thm2}, we note that in the selfadjoint case the eigenvalues of $J$ are in
$\Real\setminus [-2,2]$, and we have
$\dist(\lambda,[-2,2])=|\lambda- 2|$ for $\lambda>2$ and
$\dist(\lambda,[-2,2])=|\lambda+2|$ if $\lambda<-2$, so that
(\ref{hs2}) can be rewritten in the form
\begin{equation}\label{hs3}
\sum_{\lambda \in \sigma_d(J)} \dist(\lambda,[-2,2])^{p-\frac 1 2} \leq C(p) \|d\|_{l^p}^p.
\end{equation}
One could try to generalise (\ref{hs3}) to non-selfadjoint Jacobi
operators, but we are not able to do this, and in fact we conjecture
that (\ref{hs3}) is not true in the non-selfadjoint case, due to
a different behaviour of sequences of eigenvalues when converging to
$\pm2$ or to $(-2,2)$, respectively. To get an analogue of (\ref{hs2})
which is valid in the non-selfadjoint case, we note that since
\begin{eqnarray*}
\frac{\dist(\lambda,[-2,2])^{p}}{|\lambda^2-4|^{\frac{1}{2}}} \leq \frac 1 2 \left\{
    \begin{array}{cl}
      |\lambda-2|^{p-\frac 1 2}, \quad & \lambda > 2 \\[4pt]
      |\lambda+2|^{p-\frac 1 2}, \quad & \lambda < -2,
    \end{array}\right.
\end{eqnarray*}
inequality (\ref{hs2}) implies that in the selfadjoint case
  \begin{eqnarray}\label{rt}&&\sum_{\lambda \in \sigma_d(J)}
\frac{\dist(\lambda,[-2,2])^{p}}{|\lambda^2-4|^{\frac 1 2}} \leq   C(p)\|d\|_{l^p}^p.
\end{eqnarray}
Clearly, for $p>1$, (\ref{rt}) is the same as
(\ref{inequality2}) when $\tau=0$. On the other hand, Theorem
\ref{thm2} requires $\tau
> 0$, so that (\ref{inequality2}) (when applied to selfadjoint
operators) is slightly weaker than (\ref{rt}), which is why we say
that it is a `near-generalisation'. The same observation applies to
inequality (\ref{inequality2b}) in case that $p=1$. An interesting
open question is whether (\ref{inequality2}) and
(\ref{inequality2b}) are valid for $\tau=0$ in the non-selfadjoint
case.

We note that the methods used by Hundertmark and Simon in
\cite{hundertsimon} depend in an essential way on the
selfadjointness of the operators, and so are completely different
from those used here. Our approach develops and sharpens ideas of
Borichev, Golinskii and Kupin \cite{Borichev08}. Using determinants
of Schatten class operators, one defines a function $g(\lambda)$
whose zeros coincide with the eigenvalues of $J$, and then these
zeros are studied by applying complex function theory. Other
applications of this approach can be found in \cite{dk06,dhk}.

In \cite{Borichev08}, Theorem 2.3, the authors used this approach to show that for $p > 1$ and $\tau > 0$,
\begin{equation}\label{kupin}
\sum_{\lambda \in \sigma_{d}(J)}
\frac{\dist(\lambda,[-2,2])^{p+1+\tau}}{|\lambda^2-4|} \leq
C(\tau,\|J-J_0\|,p) \|J-J_0\|_{\mathbf{S}_p}^p,
\end{equation}
where $\|.\|_{\mathbf{S}_p}$ denotes the $p$-th Schatten norm. We
note that the Schatten norm $\|J-J_0\|_{\mathbf{S}_p}$ is equivalent
to $\|d\|_{l^p}$, as is shown in Lemma \ref{esti1} below. Inequality
(\ref{kupin}) was originally derived for Jacobi operators on
$l^2(\mn)$ but it carries over to the whole-line case. The authors
of \cite{Borichev08} also derived a more refined estimate in case
$p=1$, similar to (\ref{inequality2b}), but here their proof seems
to use special properties of the half-line
 operator and is thus not directly transferable to the whole-line setting.

We remark that inequality (\ref{kupin}) is an easy consequence of Theorem \ref{thm2} since
$$\frac{\dist(\lambda,[-2,2])^{p+1+\tau}}{|\lambda^2-4|} \leq \frac{\dist(\lambda,[-2,2])^{p+\tau}}{|\lambda^2-4|^{\frac{1}{2}}},$$
as a direct calculation shows. Theorem \ref{thm2} improves upon (\ref{kupin}) in another respect since the constants on the right-hand side of (\ref{inequality2}) and (\ref{inequality2b}) are independent of $J$.
To get rid of this dependence, we develop, in Theorem
\ref{thm:c2} below, a variant of the complex function result used in
\cite{Borichev08}. The proof of Theorem \ref{thm2} further depends on a
more subtle estimate on the function $g(\lambda)$ (mentioned above) than
the one used in \cite{Borichev08}, exploiting the
structure of the Jacobi operators in an essential way.

In relation to Theorem \ref{thm2}, it is interesting to discuss
another approach to generalising inequality (\ref{hs2}) to
non-selfadjoint operators, developed by
Golinskii and Kupin \cite{golinskii}. For $\theta \in [0,\frac{\pi} 2)$ we define the following sectors in the complex plane:
\begin{equation*}
\Omega_\theta^{\pm}= \{ \lambda : 2 \mp \Re(\lambda) < \tan(\theta) |\Im \lambda| \}.
\end{equation*}
\begin{thm_nonumber}[\cite{golinskii}, Theorem 1.5]
Let $\theta \in [0,\frac{\pi} 2)$.
Then for $p \geq \frac 3 2$
\begin{eqnarray}
\sum_{\lambda\in \sigma_{d}(J) \cap \Omega_\theta^+} |\lambda-2|^{p -
\frac 1 2} + \sum_{\lambda\in \sigma_{d}(J) \cap \Omega_\theta^-}
|\lambda+2|^{p - \frac 1 2} \leq C(p,\theta) \|d\|_{l^{p}}^{p}, \label{gk1}
\end{eqnarray}
where $C(p,\theta)=C(p)(1+2\tan(\theta))^p$.
\end{thm_nonumber}

Clearly, this theorem, when restricted to the selfadjoint case,
gives (\ref{hs2}). This is not a coincidence, but due to the fact
that its proof is obtained by a reduction to the case of selfadjoint
operators and employing (\ref{hs2}). A drawback of this result is
that the sum is not over all eigenvalues since it excludes a
diamond-shaped region around the interval $[-2,2]$ (thus avoiding sequences of eigenvalues converging to some point in $(-2,2)$).
However, we shall show that by a suitable integration, the inequalities (\ref{gk1}) can be used to derive an inequality where the sum is over all the
eigenvalues:

\begin{thm}\label{thm4}
Let $\tau \in (0,1)$. If $d \in l^p(\mz)$, where $p\geq \frac 3 2$, then
\begin{equation}\label{res2:thm4}
\sum_{\lambda \in \sigma_d(J)}
\frac{\dist(\lambda,[-2,2])^{p+\tau}}{|\lambda^2-4|^{\frac{1}{2}+\tau}} \leq
C(p,\tau)\|d\|_{l^p}^p.
\end{equation}
\end{thm}

We emphasise that the proof of Theorem \ref{thm4} does not involve any complex analysis and is thus completely different from the proof of Theorem \ref{thm2}. Let us note the similarities, and the differences, between these two results:
inequality (\ref{res2:thm4}) is in fact somewhat stronger than
(\ref{inequality2}), because of the $\tau$ in the denominator of
(\ref{res2:thm4}). However, while Theorem \ref{thm4} requires the
condition $p\geq \frac 3 2$, Theorem \ref{thm2} is valid for $p\geq
1$, just like the corresponding inequality (\ref{hs2}) in the
selfadjoint case.

We can thus conclude that the approach based on complex analysis
provides very satisfying results for non-selfadjoint Jacobi
operators, which are almost as strong as those obtained in the
selfadjoint case by specialised methods relying on the selfadjointness of the operators.

Let us give a short overview of the contents of this paper: in Section \ref{prelim}, we gather information about Schatten classes and infinite determinants. In Section \ref{complex} we present some complex analysis results that are used in the proof of Theorem \ref{thm2}, which is provided in Section \ref{proofs}. In the final Section \ref{final} we are concerned with the proof of Theorem \ref{thm4}.

\section{Preliminaries}\label{prelim}
For a Hilbert space $\mathcal{H}$ let
$\mathbf{C}(\mathcal{H})$ and $\mathbf{B}(\mathcal{H})$ denote the
classes of closed and of bounded linear operators on $\mathcal{H}$,
respectively. We denote the ideal of all compact operators on
$\mathcal{H}$ by $\mathbf{S}_\infty$ and the ideal of all Schatten
class operators by $\mathbf{S}_p, p > 0$, i.e. a compact operator $C
\in \mathbf{S}_p$ if
\begin{equation*}
  \| C\|_{\mathbf{S}_p}^p = \sum_{n=1}^\infty \mu_n(C)^p < \infty,
\end{equation*}
where $\mu_n(C)$ denotes the $n$-th singular value of $C$.

Schatten class operators obey the following H\"older's inequality:
let $p,p_1,p_2$ be positive numbers with $\frac 1 {p} = \frac 1
{p_1} + \frac 1 {p_2}$. If $C_i \in \mathbf{S}_{p_i} (i=1,2)$, then
the operator $C=C_1C_2 \in \mathbf{S}_p$ and
\begin{equation}\label{holder}
  \|C\|_{\mathbf{S}_p} \leq \| C_1\|_{\mathbf{S}_{p_1}} \| C_2\|_{\mathbf{S}_{p_2}},
\end{equation}
(see e.g. Simon \cite{simonb}, Theorem 2.8).

Let $C
\in \mathbf{S}_n$, where $n\in \mn$. Then one can define the (regularized)
determinant
$${\det}_n(I-C)=\prod_{\lambda \in \:\sigma(C)} \left[(1-\lambda) \exp \left( \sum_{j=1}^{n-1} \frac{\lambda^j}{j} \right)\right],$$
having the following properties (see e.g. Dunford and Schwartz
\cite{Dunford}, Gohberg and Kre{\u\i}n \cite{Gohberg}, or Simon
\cite{simonb}):

\noindent 1. $I-C$ is invertible if and only if $\det_n(I-C)\neq 0$.

\noindent 2. $\det_n(I)=1$.

\noindent 3.  For $A,B \in
\mathbf{B}(\mathcal{H})$ with $AB, BA \in \mathbf{S}_n$:
\begin{equation}\label{eq:comm}
  {\det}_n(I-AB)={\det}_n(I-BA).
\end{equation}

\noindent 4. If $C(\lambda)\in \mathbf{S}_n$ depends holomorphically
on $\lambda \in \Omega$, where $\Omega\subset \mc$ is open, then
$\det_n(I-C(\lambda))$ is holomorphic on $\Omega$.

\noindent 5. If $C\in \mathbf{S}_p$ for some $p>0$, then $C\in
\mathbf{S}_{\lceil p \rceil}$, where
\[ \lceil p \rceil = \min \{ n \in \mn :  n \geq p \},\]
and the following inequality holds,
\begin{equation}\label{inequality}
|{\det}_{\lceil p \rceil}(I-C)|\leq \exp \left( \Gamma_p \|C
\|_{\mathbf{S}_p}^p\right),
\end{equation}
where $\Gamma_p $ is some positive constant, see \cite[page
1106]{Dunford}. We remark that $\Gamma_p= \frac{1}{p}$ for $p \leq
1$, $\Gamma_2=\frac 1 2$ and $\Gamma_p \leq e(2+ \log p)$ in
general, see Simon \cite{Simon77}.

For $A,B \in \mathbf{B}(\mathcal{H})$ with $B-A \in \mathbf{S}_p$,
the $\lceil p \rceil$-regularized perturbation determinant of $A$ by
$B-A$ is a well defined holomorphic function on $\rho(A)=\mc \setminus
\sigma(A)$, given by
$$ d(\lambda)= {\det}_{\lceil p \rceil}(I-(\lambda-A)^{-1}(B-A)).$$
Furthermore, $\lambda_0 \in \rho(A)$ is an eigenvalue of $B$ of
algebraic multiplicity $k_0$ if and only if $\lambda_0$ is a zero of
$d(\cdot)$ of the same multiplicity.

\section{Complex analysis}\label{complex}
Let $\md= \{ z \in \mc : |z| < 1$\}. The following result is due to Borichev, Golinskii and Kupin \cite{Borichev08}.
\begin{thm}\label{thm:c1}
Let $h: \md \to \mc$ be holomorphic with $h(0)=1$, and suppose that
$$\log|h(z)|\leq K\frac{1}{(1-|z|)^{\alpha}} \prod_{j=1}^N \frac{1}{|z-\xi_j|^{\beta_j}},$$
where $|\xi_j|=1$ ($1\leq j\leq N$), and the exponents
$\alpha,\beta_j$ are nonnegative. Let $\tau>0$. Then the zeros of $h$
satisfy the inequality
\begin{equation*}
\sum_{z \in \md, h(z)=0}(1-|z|)^{\alpha+1+\tau}\prod_{j=1}^N|z-\xi_j|^{(\beta_j-1+\tau)_+}\leq
C(\alpha,\{ \beta_j\},\{\xi_j\}, \tau)K,\end{equation*}
where $x_+=\max(x,0)$, and each zero of $h$ is counted according to its multiplicity.
\end{thm}
For the holomorphic function $h$ that we will consider below,  there will be some additional information
on the speed of convergence of $ \log|h(z)| \to 0$ as $z \to 0$. The following modification of the above theorem takes this
into account.
\begin{thm}\label{thm:c2}
Let $h : \md \to \mc$ be holomorphic with $h(0)=1$, and suppose that
\begin{equation}\label{eq:c2}
  \log |h(z)| \leq K\frac{|z|^\gamma}{(1-|z|)^{\alpha}} \prod_{j=1}^N \frac{1}{|z-\xi_j|^{\beta_j}},
\end{equation}
where $|\xi_j|=1 \;(1 \leq j \leq N)$, and the exponents $\alpha, \beta_j$ and $\gamma$ are nonnegative.
Let $0 < \tau < 1$. Then the zeros of $h$ satisfy the inequality
\begin{small}
\begin{equation}\label{eq:c3}
\sum_{z\in \md, h(z)=0}
\frac{(1-|z|)^{\alpha+1+\tau}}{|z|^{(\gamma-1+\tau)_+}}\prod_{j=1}^N|z-\xi_j|^{(\beta_j-1+\tau)_+} \leq
C(\alpha,\{\beta_j\},\gamma, \{\xi_j\},\tau)K,
\end{equation}
\end{small}
where each zero is counted according to its multiplicity.
\end{thm}
We note that the results of the previous theorem
differ from the results  obtained in Theorem \ref{thm:c1} both in the hypothesis, which
requires $\log|h(z)|$ to vanish at $0$ at a specified rate, and
requires $\tau<1$, and in its conclusion, which includes the
$\frac{1}{|z|}$ term.
\begin{proof}[Proof of Theorem \ref{thm:c2}]
In view of Theorem \ref{thm:c1}, we only need to consider the case $\gamma > 1-\tau$. Set $\md_r=\{ z \in \mc : |z|<r \}$.  Using the boundedness of $\frac 1
{|z|}$ on $\md \setminus \md_{\frac 1 2}$ , we obtain from
Theorem \ref{thm:c1} that
\begin{small}
\begin{eqnarray*}
\sum_{z \in \md \setminus \md_{\frac 1 2},\; h(z)=0} \frac{(1-|z|)^{\alpha+1+\tau}}{|z|^{\gamma-1+\tau}}\prod_{j=1}^N|z-\xi_j|^{(\beta_j-1+\tau)_+}
 \leq
C(\alpha,\{\beta_j\},\gamma, \{\xi_j\}, \tau)K.
\end{eqnarray*}
\end{small}
Hence, the proof is completed by showing that
\begin{equation}\label{eq:c5}
\sum_{z \in \md_{\frac 1 2},\; h(z)=0} \frac{1}{|z|^{\gamma-1+\tau}}
\leq  C(\alpha,\{\beta_j\},\gamma, \{\xi_j\},\tau) K.
\end{equation}
Let $N_h(\md_r)$ denote the number of zeros of $h$ in $\md_r$ (multiplicities taken into account). Then we can rewrite the sum in (\ref{eq:c5}) as follows:
\begin{eqnarray}
&&\sum_{z \in \md_{\frac 1 2},\; h(z)=0} \frac{1}{|z|^{\gamma-1+\tau}}
= (\gamma-1+\tau) \sum_{z \in \md_{\frac 1 2},\; h(z)=0} \int_0^{\frac 1 {|z|}} dt \; t^{\gamma-2+\tau} \nonumber\\
&=& (\gamma-1+\tau) \int_0^{\infty} dt\; t^{\gamma-2+\tau} \left( \sum_{|z| < \min(\frac 1 2, t^{-1}),\; h(z)=0} 1 \right) \nonumber \\
&=& (\gamma-1+\tau) \left[ \int_0^2 dt\; t^{\gamma-2+\tau} N_h(\md_{\frac 1 2}) +  \int_2^\infty dt\; t^{\gamma-2+\tau} N_h(\md_{t^{-1}})  \right].
\label{eq:c6}
\end{eqnarray}
To estimate the last two integrals the following lemma is used.
\begin{lemma}\label{lem:c1}
Assume (\ref{eq:c2}). Then for $r \in (0,\frac 1 2]$ we have
\begin{eqnarray*}
  N_h(\md_r) \leq C(\alpha,\{\beta_j\}, \{\xi_j\})K r^{\gamma}  .
\end{eqnarray*}
\end{lemma}
\begin{proof}[Proof of Lemma \ref{lem:c1}]
Let $0 < r < s < 1$. From Jensen's identity (see e.g. Rudin
\cite{Rudin87}, Theorem 15.18) and assumption (\ref{eq:c2}) we obtain
\begin{eqnarray*}
N_h(\md_r) &=& \frac{1}{\log (\frac s r) } \sum_{z \in \md_r, h(z)=0} \log \left( \frac s r \right)
\leq \frac{1}{\log (\frac s r) } \sum_{z \in \md_r, h(z)=0} \log \left(  \frac s {|z|} \right) \\
&\leq& \frac{1}{\log (\frac s r) } \sum_{z \in \md_s, h(z)=0} \log \left(  \frac s {|z|} \right)
= \frac{1}{\log (\frac s r) } \frac{1}{2\pi} \int_0^{2\pi} \log |h(s e^{i\theta})| d\theta  \\
&\leq& \frac{1}{\log (\frac s r) }  \frac{K
s^\gamma}{(1-s)^\alpha} \frac{1}{2\pi} \int_0^{2\pi}
\prod_{j=1}^N \frac{1}{ | s e^{i\theta} - \xi_j|^{\beta_j}} d\theta .
\end{eqnarray*}
Choosing $s = \frac{3}{2} r$ (i.e. $s \leq \frac 3 4$) concludes the proof of the lemma.
\end{proof}
Returning to (\ref{eq:c6}), we can use Lemma \ref{lem:c1} and the
fact that $\gamma > 1 - \tau$ to conclude
\begin{eqnarray*}
  \int_0^2 dt \; t^{\gamma-2+\tau} N_h(\md_{\frac 1 2}) 
&\leq& C(\alpha,\{\beta_j\},\gamma, \{\xi_j\}, \tau) K.
\end{eqnarray*}
Similarly, using that $\tau<1$, Lemma \ref{lem:c1} implies that
\begin{eqnarray*}
\int_2^\infty dt \; t^{\gamma-2+\tau} N_h(\md_{t^{-1}})
&\leq& C(\alpha,\{\beta_j\}, \{\xi_j\}) K  \int_2^\infty dt \; t^{-2+\tau} \\
&\leq&
C(\alpha,\{\beta_j\},\gamma, \{\xi_j\}, \tau) K.
\end{eqnarray*}
This concludes the proof of Theorem \ref{thm:c2}.
\end{proof}

We now translate the result of Theorem \ref{thm:c2} into a result about
holomorphic functions on $\Complex\setminus [-2,2]$, which is the
one we will use below.
\begin{cor}\label{cor:c1}
 Let $g : \mc \setminus [-2,2] \to \mc$ be holomorphic with $\lim_{|\lambda| \to \infty} g(\lambda)=1$. For $\alpha, \beta \geq 0$ suppose that
\begin{equation}\label{eq:c8}
  \log |g(\lambda)| \leq \frac{K_0}{\dist(\lambda,[-2,2])^{\alpha}|\lambda^2-4|^{\beta} }.
\end{equation}
Let $0 < \tau < 1$ and set
\begin{eqnarray}\label{eq:c8b}
\begin{array}{clc}
  \eta_1 &=& \alpha + 1 +\tau, \\[4pt]
  \eta_2 &=& (2\beta+\alpha-1+\tau)_+.
\end{array}
\end{eqnarray}
Then,
\begin{equation}\label{eq:c9}
\sum_{\lambda \in \mc \setminus [-2,2], g(\lambda)=0}
\frac{\dist(\lambda,[-2,2])^{\eta_1}}{|\lambda^2-4|^{\frac{\eta_1-\eta_2}{2}}} \leq
C(\alpha,\beta,\tau)K_0,
\end{equation}
where each zero of $g$ is counted according to its multiplicity
\end{cor}

\begin{proof}
We note that $z \mapsto z + z^{-1}$ maps $\md \setminus \{0 \}$ conformally onto $\mc \setminus [-2,2]$. So we can define
$$h(z)=g(z + z^{-1}), \quad z \in \md \setminus \{ 0 \}.$$ Setting $h(0)=1$, $h$ is holomorphic on $\md$ and
 \begin{equation}
   \log | h(z)| \leq \frac{K_0}{\dist(\lambda,[-2,2])^{\alpha}|\lambda^2-4|^{\beta} }, \label{eq:c9b}
 \end{equation}
where $\lambda=z+z^{-1}$. To express the right-hand side of the last equation in terms of $z$, the following Lemma is needed. Its proof is given below.
\begin{lemma}\label{lem:ko1}
For $\lambda=z+z^{-1}$ with $ z \in \md$, we have
  \begin{equation}\label{eq:ko1}
 \frac{1}{2} \frac{|z^2-1|(1-|z|)}{|z|} \leq \dist(\lambda,[-2,2]) \leq   \frac{1+\sqrt{2}}{2} \frac{|z^2-1|(1-|z|)}{|z|}.
  \end{equation}
\end{lemma}
In addition to the last lemma, a direct calculation shows that
\begin{equation}\label{eq:c11}
  |\lambda^2-4| = \left| \frac{z^2-1}{z} \right|^2.
\end{equation}
Using (\ref{eq:ko1}) and (\ref{eq:c11}), we can estimate (\ref{eq:c9b}) as follows
\begin{equation*}
  \log | h(z)| \leq \frac{2^{\alpha}K_0 |z|^{2\beta+\alpha}}{(1-|z|)^\alpha |z^2-1|^{2\beta+\alpha}}.
\end{equation*}
Hence, Theorem \ref{thm:c2} implies that for $0<\tau < 1$ and $\eta_1, \eta_2$ as defined in (\ref{eq:c8b}),
\begin{equation*}
\sum_{z \in \md, h(z)=0}
(1-|z|)^{\eta_1} \left| \frac{z^2-1}{z} \right|^{\eta_2} \leq
C(\alpha, \beta,\tau)K_0.
\end{equation*}
By (\ref{eq:ko1}) and (\ref{eq:c11}),
\begin{eqnarray*}
(1-|z|)^{\eta_1} \left| \frac{z^2-1}{z} \right|^{\eta_2} \geq \left( \frac 2 {1 + \sqrt{2}} \right)^{\eta_1}
\frac{ \dist(\lambda,[-2,2])^{\eta_1}}{|\lambda^2-4|^{\frac{\eta_1-\eta_2}{2}}}.
\end{eqnarray*}
This concludes the proof of Corollary \ref{cor:c1}.
\end{proof}
\begin{proof}[Proof of Lemma \ref{lem:ko1}]
Let $z\in\md$ and set $\lambda=z+z^{-1}$. We define $G_1=\{ z : \Re(\lambda) \leq -2\}$, $G_2= \{ z : \Re(\lambda) \geq 2\}$ and $G_3= \{ z : |\Re(\lambda)| < 2\}$.
Then
\begin{eqnarray}
  \dist(\lambda,[-2,2]) &=& \left\{
    \begin{array}{cl}
      |\lambda+2|=   \frac{|1+z|^2}{|z|}, &  z \in G_1 \\[2pt]
      |\lambda-2|=   \frac{|1-z|^2}{|z|}, &  z \in G_2 \\[2pt]
      |\Im \lambda|= |\Im z| \frac{1-|z|^2}{|z|^2}, &  z \in G_3.
    \end{array}\right. \label{eq:di}
\end{eqnarray}
We first show that for $z \in G_3$ the following holds
\begin{equation}\label{eq:ay}
  \frac{1}{\sqrt{2}} \frac{|z^2-1|(1-|z|)}{|z|} \leq \dist(\lambda,[-2,2]) \leq \frac{1+\sqrt{2}}{2} \frac{|z^2-1|(1-|z|)}{|z|}.
\end{equation}
By (\ref{eq:di}), this is equivalent to
\begin{equation}\label{eq:aax}
\frac 1 {\sqrt{2}} \leq   {|\Im z|} \frac{1+|z|}{|z||z^2-1|} \leq \frac{1+\sqrt{2}}{2}.
\end{equation}
Switching to polar coordinates we see that  $re^{i \theta} \in G_3$ if $\cos^2(\theta) <  \frac {4r^2} {(1+r^2)^2}$ and (\ref{eq:aax}) can be rewritten as
\begin{equation}\label{eq:as1}
\frac 1 {\sqrt{2}} \leq  \frac{(1+r)\sqrt{1-\cos^2(\theta)}}{\sqrt{(1+r^2)^2-4{r^2}\cos^2(\theta)}} \leq \frac{1+\sqrt{2}}{2}.
\end{equation}
For $x=\cos^2(\theta)$ and fixed $r$ we define
\begin{equation*}\label{eq:kl2}
f(x)= \frac{{1-x}}{{(1+r^2)^2-4{r^2}x}} \quad , 0 \leq x <  \frac {4r^2} {(1+r^2)^2}.
\end{equation*}
It is easy to see that $f$ is monotonically decreasing. We thus obtain
\begin{equation*}\label{eq:ad}
\frac{1}{1+6r^2+r^4}= f\left( \frac {4r^2} {(1+r^2)^2} \right) \leq  f(x) \leq f(0) = \frac{1}{(1+r^2)^2}.
\end{equation*}
The last chain of inequalities implies  the validity of (\ref{eq:as1}) and (\ref{eq:aax}) since
\[ \sup_{r \in [0,1]} \frac{1+r}{1+r^2}=\frac{1+\sqrt{2}}{2} \quad \text{and} \quad \inf_{r \in [0,1]} \frac{1+r}{\sqrt{1+6r^2+r^4}}= \frac 1 {\sqrt{2}}. \]
Next, we show that the following holds for  $z \in G_1 \cup G_2$,
\begin{equation*}\label{eq:az}
\frac 1 2 \frac{|z^2-1|(1-|z|)}{|z|} \leq \dist(\lambda,[-2,2]) \leq \frac{1+\sqrt{2}}{2} \frac{|z^2-1|(1-|z|)}{|z|}.
\end{equation*}
By symmetry, it is sufficient to show it for $z \in G_1$. In this case, by (\ref{eq:di}), the last inequality is equivalent to
\begin{equation}\label{eq:rt}
  \frac 1 2 \leq \frac{|z+1|}{|z-1|(1-|z|)} \leq \frac{1+\sqrt{2}}{2}.
\end{equation}
Switching to polar coordinates , we have to show that
\begin{equation}\label{eq:wk}
  \frac 1 2  \leq  \frac{1}{1-r} \sqrt{ \frac{r^2+1+2r\cos(\theta)}{r^2+1-2r\cos(\theta)}} \leq \frac{1+\sqrt{2}}{2},
\end{equation}
where $\cos(\theta) \leq  \frac {-2r} {1+r^2}$. For $y=\cos(\theta)$ and fixed $r$ we define
\begin{equation}\label{eq:kj}
  \tilde{f}(y)=\frac{r^2+1+2ry}{r^2+1-2ry} \quad, -1 \leq y \leq  \frac {-2r} {1+r^2}.
\end{equation}
A short calculation shows that $\tilde{f}$ is monotonically increasing and we obtain that
\begin{equation}\label{eq:kl}
  \left( \frac{ 1-r}{1+r} \right)^2= \tilde{f}(-1) \leq \tilde{f}(y) \leq \tilde{f} \left(  \frac {-2r} {1+r^2} \right) =  \frac{(1-r^2)^2}{1+6r^2+r^4} .
\end{equation}
(\ref{eq:kj}) and (\ref{eq:kl}) imply the validity of (\ref{eq:wk}) and (\ref{eq:rt}) since
\begin{equation}
  \inf_{r \in [0,1]} \frac{1}{1+r} = \frac 1 2 \quad \text{ and } \quad \sup_{r \in [0,1]}  \frac{1+r}{\sqrt{1+6r^2+r^4}} \leq \frac{1+\sqrt{2}}{2}.
\end{equation}
This concludes the proof of Lemma \ref{lem:ko1}.
\end{proof}

\section{Proof of Theorem \ref{thm2} }\label{proofs}

We denote the standard basis of $l^2(\Integer)$  by
$\{\delta_k\}_{k\in \Integer}$, i.e. $\delta_k(k)=1$ and
$\delta_k(j)=0$ for $j\neq k$. Let the diagonal operator $D \in
\mathbf{B}(l^2(\mz))$ be defined via $D\delta_k=d_k \delta_k$, where
the  sequence $d=\{d_k\}$ is as defined in (\ref{defdk}), i.e.
$$ d_k=\max\Big(|a_{k-1}-1|,|a_{k}-1|,|b_k|,|c_{k-1}-1|,|c_{k}-1|\Big).$$
Furthermore, the operator $U \in \mathbf{B}(l^2(\mz))$ is given by
$$U \delta_k=u_k^- \delta_{k-1} +
u_k^0\delta_k +u_k^+\delta_{k+1},$$ where
$$u_k^-=\frac{c_{k-1}-1}{\sqrt{d_{k-1}d_k}},\quad
u_k^0=\frac{b_k}{d_k},\quad u_k^+=\frac{a_{k}-1}{\sqrt{d_{k+1}d_k}}.$$
Here we use the convention $\frac{0}{0}=1$. It is then easily checked that
\begin{equation}\label{rj}J-J_0=D^{\frac 1 2} U D^{\frac 1 2}.\end{equation}
Moreover, the definition of $\{d_k\}$ implies
$$|u_k^-|<1,\;\;|u_k^0|<1,\;\;|u_k^+|<1,$$
showing  that $\|U\|\leq 3$.

The following lemma is a variation on Theorem 2.4 of
\cite{killipsimon}.
\begin{lemma}\label{esti1}
  Let $p\geq 1$ and $d \in l^p(\mz)$. Then $J-J_0 \in \mathbf{S}_p$ and
  \begin{equation}
   \frac{1}{6^{\frac{1}{p}}} \|d\|_{l^p}\leq   \|J-J_0\|_{\mathbf{S}_p} \leq 3 \|d\|_{l^p}.
  \end{equation}
\end{lemma}
\begin{proof}
 From (\ref{rj}) we obtain
 \begin{eqnarray*}
   \|J-J_0\|_{\mathbf{S}_p} &=& \|D^{\frac 1 2}UD^{\frac 1 2}\|_{\mathbf{S}_p} \leq \|D^{\frac 1 2}\|_{\mathbf{S}_{2p}} \|UD^{\frac 1 2}\|_{\mathbf{S}_{2p}}  \\
&\leq& \|U\|\|D^{\frac 1 2}\|_{\mathbf{S}_{2p}}^2 \leq 3 \|D^{\frac 1 2}\|_{\mathbf{S}_{2p}}^2 = 3 \|d\|_{l^p}.
 \end{eqnarray*}
In the first estimate we applied H\"older's inequality for Schatten norms, see (\ref{holder}). The last equality is valid since the diagonal operator $D^{\frac 1 2}$ is selfadjoint and nonnegative with
eigenvalues $d_k^{\frac 1 2}$.

For the inequality in the other direction, we use that (see
\cite{killipsimon}, Lemma 2.3(iii))
$$\sum_{k\in \Integer}[|a_k-1|^p+|c_k-1|^p+|b_k|^p]\leq 3 \|J-J_0\|_{{\bf{S}}^p}^p.$$
Since
\begin{eqnarray*}
\|d\|_{l^p}^p &=& \sum_{k\in \Integer}\max\Big(|a_{k-1}-1|^p,|a_{k}-1|^p,|b_k|^p,|c_{k-1}-1|^p,|c_{k}-1|^p\Big) \\
&\leq& 2\sum_{k\in \Integer}[|a_k-1|^p+|c_k-1|^p+|b_k|^p],
\end{eqnarray*}
the result follows.
\end{proof}
In the sequel, we suppose that  $d\in l^p(\mz)$ for some fixed $p\geq1$.
Using the results of Section \ref{prelim}, we can define the holomorphic function $g : \mc \setminus [-2,2] \to \mc$ via
\begin{equation}
 g(\lambda) = {\det}_{\lceil p \rceil}(I-(\lambda-J_0)^{-1}(J-J_0)), \label{eq:def_g1}
\end{equation}
and the zeros of $g$ coincide with the eigenvalues of $J$ in $\mc \setminus [-2,2]$, where multiplicity is taken into account. We further note that $\lim_{|\lambda| \to \infty}g(\lambda)=1$.

For $\lambda \in \mc \setminus [-2,2]$  we define
\begin{equation}\label{G}
G(\lambda)=D^{\frac 1 2}(\lambda-J_0)^{-1}D^{\frac 1 2}.
\end{equation}
We will see below that $G(\lambda) \in \mathbf{S}_p$. Hence, (\ref{eq:comm}) and (\ref{rj}) allow to derive the following alternative representation of $g$.
\begin{eqnarray}
  g(\lambda)&=& {\det}_{\lceil p \rceil}(I-(\lambda-J_0)^{-1}D^{\frac 1 2}UD^{\frac 1 2}) \nonumber \\
  &=& {\det}_{\lceil p \rceil}(I-G(\lambda)U).
\end{eqnarray}
From (\ref{inequality}) we further obtain that
\begin{equation}\label{Q1}
  \log | g(\lambda)| \leq \Gamma_p \| G(\lambda)U\|_{\mathbf{S}_p}^p \leq \Gamma_p 3^p \| G(\lambda)\|_{\mathbf{S}_p}^p.
\end{equation}
The following lemma provides some information on the Schatten norm of $G(\lambda)$.
\begin{lemma}\label{lem:est_g}
 Let $d \in l^p(\mz),$ where $p \geq 1$. Then $G(\lambda) \in \mathbf{S}_p$ and for $p>1$,
 \begin{equation}\label{Q2}
   \|  G(\lambda)\|_{\mathbf{S}_p}^p \leq  \frac{C(p) \|d\|_{l^p}^p}{\dist(\lambda,[-2,2])^{p-1}|\lambda^2-4|^{\frac 1 2}}.
 \end{equation}
Furthermore, for every $0<\eps<1$,
 \begin{equation}\label{Q3}
   \|  G(\lambda)\|_{\mathbf{S}_1} \leq \frac{C(\eps)\|d\|_{l^1} }{\dist(\lambda,[-2,2])^{\eps}|\lambda^2-4|^{\frac{(1-\eps)} 2}}.
 \end{equation}
\end{lemma}
The proof of Lemma \ref{lem:est_g} will be given below. First, let us continue with the proof of Theorem \ref{thm2}.
Fix $\tau \in (0,1)$. We consider the case $p>1$ first. From (\ref{Q1}) and (\ref{Q2}) we obtain that
 \begin{eqnarray*}
   \log|g(\lambda)|\leq \frac{C(p)\|d\|_{l^p}^p }{\dist(\lambda,[-2,2])^{p-1}|\lambda^2-4|^{\frac 1 2}}.
 \end{eqnarray*}
Hence, we can  apply Corollary \ref{cor:c1} with $\alpha=p-1$ and $\beta = \frac 1 2$ (i.e. $\eta_1=p+\tau$ and $\eta_2=p-1+\tau$) to obtain
\begin{equation}
\sum_{\lambda \in \mc\setminus[-2,2],g(\lambda)=0}
\frac{\dist(\lambda,[-2,2])^{p+\tau}}{|\lambda^2-4|^{\frac{1}{2}}} \leq
C(p,\tau)\|d\|_{l^p}^p.
\end{equation}
Noting that the eigenvalues of $J$ coincide with the zeros of $g$ concludes the proof in case that $p>1$.
In case $p=1$, we obtain from (\ref{Q1}) and (\ref{Q3}) that
for $0<\eps<1$,
 \begin{eqnarray*}
   \log|g(z)|\leq \frac{C(\eps)\|d\|_{l^1} }{\dist(\lambda,[-2,2])^{\eps}|\lambda^2-4|^{\frac{(1-\eps)} 2}}.
 \end{eqnarray*}
As above, an application of Corollary \ref{cor:c1} shows that for every $\tilde{\tau} \in (0,1)$
\begin{equation}
  \sum_{\lambda \in \sigma_d(J)}
\frac{\dist(\lambda,[-2,2])^{1+\eps+\tilde{\tau}}}{|\lambda^2-4|^{\frac {(1+\eps)} 2}} \leq
C(\tilde\tau,\eps)\|d\|_{l^1}.
\end{equation}
Choosing $\eps=\tilde{\tau}=\frac \tau 2$ concludes the proof of Theorem \ref{thm2}.

\subsection{Proof of Lemma \ref{lem:est_g}}
We define the Fourier transform $\mathcal{F}:l^2(\mz) \to
L^2(0,2\pi)$ by
$$(\mathcal{F}f)(\theta)= \frac 1 {\sqrt{2\pi}} \sum_k
e^{ik\theta}f_k$$
and note that for $f \in l^2(\mz)$ and $\theta \in [0,2\pi)$:
$$(\mathcal{F}J_0 f)(\theta) = 2\cos(\theta) (\mathcal{F}f)(\theta).$$
 Consequently, for $\lambda \in \Complex\setminus [2,2]$,
  \begin{eqnarray*}
    (\lambda-J_0)^{-1}= \mathcal{F}^{-1}M_{v_\lambda}\mathcal{F},
  \end{eqnarray*}
where $M_{v_\lambda} \in \mathbf{B}(L^2(0,2\pi))$ is the operator of multiplication with the bounded function
\begin{equation}
  v_\lambda(\theta)= (\lambda-2\cos(\theta))^{-1}, \quad \theta \in [0,2\pi).\label{Q4}
\end{equation}
Since
$$ v_{\lambda}=|v_{\lambda}|^{\frac 1 2} \cdot \frac{v_\lambda}{|v_{\lambda}|} \cdot |v_{\lambda}|^{\frac 1 2},$$
we can define the unitary operator  $T=\mathcal{F}^{-1} M_{\frac{v_{\lambda}}{|v_{\lambda}|}} \mathcal{F}$, to obtain the following identity
\begin{eqnarray*}
(\lambda-J_0)^{-1}= \mathcal{F}^{-1} M_{|v_\lambda|^{\frac 1 2}} \mathcal{F} T \mathcal{F}^{-1} M_{|v_\lambda|^{\frac 1 2}} \mathcal{F}.
\end{eqnarray*}
With definition (\ref{G}) and H\"older's inequality, see (\ref{holder}), we obtain
\begin{eqnarray}
  \| G(\lambda)\|_{\mathbf{S}_p}^p &=& \| D^{\frac 1 2} \mathcal{F}^{-1} M_{|v_\lambda|^{\frac 1 2}} \mathcal{F} T \mathcal{F}^{-1} M_{|v_\lambda|^{\frac 1 2}} \mathcal{F} D^{\frac 1 2}\|_{\mathbf{S}_p}^p  \nonumber  \\
 &\leq& \| D^{\frac 1 2} \mathcal{F}^{-1} M_{|v_\lambda|^{\frac 1 2}} \mathcal{F}\|_{\mathbf{S}_{2p}}^p \|T  \mathcal{F}^{-1} M_{|v_\lambda|^{\frac 1 2}} \mathcal{F}D^{\frac 1 2}\|_{\mathbf{S}_{2p}}^p  \nonumber  \\
 &\leq& \| D^{\frac 1 2} \mathcal{F}^{-1} M_{|v_\lambda|^{\frac 1 2}} \mathcal{F}\|_{\mathbf{S}_{2p}}^p \|\mathcal{F}^{-1} M_{|v_\lambda|^{\frac 1 2}} \mathcal{F}D^{\frac 1 2}\|_{\mathbf{S}_{2p}}^p \nonumber \\
 &=& \| D^{\frac 1 2} \mathcal{F}^{-1} M_{|v_\lambda|^{\frac 1 2}} \mathcal{F}\|_{\mathbf{S}_{2p}}^{2p}. \label{R1}
\end{eqnarray}
For the last identity we used the selfadjointness of $D^{\frac 1 2}$ and $\mathcal{F}^{-1} M_{|v_\lambda|^{\frac 1 2}} \mathcal{F}$, and the fact that the Schatten norm is invariant under taking the adjoint. To derive an estimate on the Schatten norm of $D^{\frac 1 2} \mathcal{F}^{-1} M_{|v_\lambda|^{\frac 1 2}} \mathcal{F}$, we will use the following lemma. Here, as above, we denote the diagonal operator  corresponding to a sequence $k=\{k_m\} \in l^\infty(\mz)$ by $K$, i.e. $K\delta_m = k_m \delta_m$.
\begin{lemma} \label{lem:simon}
Let $q \geq 2$. Suppose that $k=\{k_m\} \in l^q(\mz)$ and $v \in L^q(0,2\pi)$. Then the following holds,
\begin{equation}
  \|K \mathcal{F}^{-1} M_v \mathcal{F} \|_{\mathbf{S}_q} \leq (2\pi)^{-1/q} \|k\|_{l^q} \|v\|_{L^q}.
\end{equation}
\end{lemma}
For operators on $L^2(\mr^d)$, this is a well-known result, we refer to Theorem 4.1 in Simon \cite{simonb}. Since the proofs in the discrete and continuous settings are completely analogous, we only provide a sketch.
\begin{proof}[Sketch of proof of Lemma \ref{lem:simon}]
We note that $K \mathcal{F}^{-1} M_v \mathcal{F}$ is an integral operator on $l^2(\mz)$ with kernel $(2\pi)^{-\frac 1 2}k_m (\mathcal{F}^{-1}v)_{m-n}$ where $m, n \in \mz$. We thus obtain for the Hilbert-Schmidt norm
\begin{eqnarray*}
\| K \mathcal{F}^{-1} M_v \mathcal{F} \|_{\mathbf{S}_2}^2 &=& (2\pi)^{-1}\sum_{m,n} |k_m (\mathcal{F}^{-1}v)_{m-n}|^2 \\
&=& (2\pi)^{-1} \|k\|_{l^2}^2 \|\mathcal{F}^{-1}v\|_{l^2}^2 = (2\pi)^{-1} \|k\|_{l^2}^2 \|v\|_{L^2}^2.
\end{eqnarray*}
Clearly, for the operator norm we have
$$ \| K \mathcal{F}^{-1} M_v \mathcal{F} \| \leq \|K\| \|\mathcal{F}^{-1} M_v \mathcal{F} \| = \|K\| \|M_v\|
 = \|k\|_{l^\infty} \|v\|_{L^\infty}.$$
The general result now follows by complex interpolation. For details, see the proof of Theorem 4.1 in \cite{simonb}.
\end{proof}
We return to the proof of Lemma \ref{lem:est_g}. With $d^{\frac 1 2}=\{ d_k^{\frac 1 2}\}$  the previous Lemma and estimate (\ref{R1}) imply
\begin{eqnarray*}
  \| G(\lambda)\|_{\mathbf{S}_p}^p \leq (2\pi)^{-1} \| d^{\frac 1 2}\|_{l^{2p}}^{2p} \| |v_\lambda|^{\frac 1 2} \|_{L^{2p}}^{2p}
  = (2\pi)^{-1} \| d \|_{l^p}^p \| v_\lambda \|_{L^p}^p.
\end{eqnarray*}
The proof of Lemma \ref{lem:est_g} is completed by an application of the following result.
\begin{lemma}\label{lem:resolv1}
  Let $\lambda \in \mc \setminus [-2,2]$ and let $v_\lambda: [0,2\pi) \to \mc$ be defined by (\ref{Q4}). Then for $p>1$,
  \begin{equation}\label{eq:resolv1}
    \| v_\lambda \|_{L^p}^p \leq \frac{C(p)}{\dist(\lambda,[-2,2])^{p-1}|\lambda^2-4|^{\frac 1 2 }}.
  \end{equation}
Furthermore, for every $0 < \eps < 1$,
\begin{equation}\label{eq:resolv2}
     \| v_\lambda \|_{L^1} \leq \frac{C(\eps)}{\dist(\lambda,[-2,2])^{\eps}|\lambda^2-4|^{\frac {(1-\eps)} 2 }}.
\end{equation}
\end{lemma}
\begin{proof}[Proof of Lemma \ref{lem:resolv1}]
Let us first show that  (\ref{eq:resolv2}) is an immediate consequence of (\ref{eq:resolv1}): for $r>1$ H\"older's inequality and (\ref{eq:resolv1}) imply (remember that $L^2=L^2(0,2\pi)$)
 \begin{eqnarray*}
   \|v_\lambda\|_{L^1} &=& \|v_\lambda \cdot 1\|_{L^1} \leq \|v_\lambda\|_{L^r} \|1\|_{L^{r/(r-1)}} \\
   &\leq& \frac{C(r)}{\dist(\lambda,[-2,2])^{1-\frac 1 r}|\lambda^2-4|^{\frac 1 {2r}}}.
 \end{eqnarray*}
Choosing $r=\frac{1}{1-\eps}$, where $0 < \eps < 1$, implies the validity of (\ref{eq:resolv2}). It remains to show (\ref{eq:resolv1}):

Let $ p > 1$ and $\lambda \in \mc \setminus [-2,2]$. Substituting $x=2\cos(\theta)$ we see that
\begin{eqnarray*}
\|v_\lambda\|_{L^p}^p &=& \int_0^{\pi}\frac{d\theta}{|\lambda-2\cos(\theta)|^p} + \int_\pi^{2\pi}\frac{d\theta}{|\lambda-2\cos(\theta)|^p}  \\
&=&2\int_{-2}^2 \frac{dx}{|\lambda-x|^p(4-x^2)^{\frac{1}{2}}} .
\end{eqnarray*}
Simplifying further we obtain
\begin{eqnarray}
\| v_{\lambda}\|_{L^p}^p &=&2 \left( \int_{0}^2 \frac{dx}{|\lambda-x|^p(4-x^2)^{\frac{1}{2}}} + \int_{0}^2 \frac{dx}{|\lambda+x|^p(4-x^2)^{\frac{1}{2}}} \right). \label{iE}
\end{eqnarray}
The last equality shows that the $L^p$-norm of $v_\lambda$ is
invariant under reflection of $\lambda$ with respect to the real and
the imaginary axes, respectively. Since the same is true for the
right-hand side of (\ref{eq:resolv1}), we can restrict ourselves to
the case that $\Re(\lambda)$ and $\Im(\lambda)$ are both
nonnegative. In this case, using that for $x \in [0,2]$ we have $
|\lambda + x| \geq |\lambda - x|$, we can deduce from (\ref{iE})
that
\begin{eqnarray*}
 \| v_\lambda \|_{L^p}^p \leq 4 \int_{0}^2 \frac{dx}{|\lambda-x|^p(4-x^2)^{\frac{1}{2}}}
\leq  2\sqrt{2} \int_{0}^2 \frac{dx}{|\lambda-x|^p(2-x)^{\frac{1}{2}}}. 
\end{eqnarray*}
In the following we set $\lambda_0=\Re(\lambda)$ and $\lambda_1=\Im(\lambda)$. The previous considerations imply that it is sufficient to show that for $\lambda_0, \lambda_1 \geq 0$:
\begin{equation}\label{eq:gg}
\int_{0}^2 \frac{dx}{|\lambda-x|^p(2-x)^{\frac{1}{2}}} \leq \frac{C(p)}{\dist(\lambda,[-2,2])^{p-1}|\lambda^2-4|^{\frac 1 2}}.
\end{equation}
We now consider seperately the cases a) $\lambda_0\geq 2$ and
b) $\lambda_0\in [0,2)$.

a.1) If $\lambda_0 \geq 2$ and $|\lambda|\geq 3$ then
\begin{eqnarray}\label{e0}
&& \int_{0}^2 \frac{dx}{|\lambda-x|^p(2-x)^{\frac{1}{2}}}\leq \frac{1}{|\lambda -2|^p}
\int_{0}^2 \frac{dx}{(2-x)^{\frac{1}{2}}}=\frac{C}{|\lambda
-2|^p} \nonumber \\
&=& \frac{C}{\dist(\lambda,[-2,2])^{p-1}|\lambda-2|}=
\frac{C }{\dist(\lambda,[-2,2])^{p-1}|\lambda^2-4|^{\frac{1}{2}}}
\frac{|\lambda+2|^{\frac{1}{2}}}{|\lambda-2|^{\frac{1}{2}}}\nonumber \\
&\leq& \frac{C}{\dist(\lambda,[-2,2])^{p-1}|\lambda^2-4|^{\frac{1}{2}}}.
\end{eqnarray}
a.2) If $\lambda_0 \geq 2$ and $|\lambda| < 3$ then
\begin{eqnarray}\label{e1}&&
\int_{0}^2 \frac{dx}{|\lambda-x|^p(2-x)^{\frac{1}{2}}}
=\int_{0}^2 \frac{dx}{((x-\lambda_0)^2+\lambda_1^2)^{\frac{p}{2}}(2-x)^{\frac{1}{2}}}\nonumber\\
&=& \int_{0}^2 \frac{dx}{((\lambda_0-2+x)^2+\lambda_1^2)^{\frac{p}{2}}x^{\frac{1}{2}}}
\leq \int_{0}^2 \frac{dx}{(x^2+(\lambda_0-2)^2+\lambda_1^2)^{\frac{p}{2}}x^{\frac{1}{2}}} \nonumber \\
&=& \int_{0}^2 \frac{dx}{(x^2+|\lambda-2|^2)^{\frac{p}{2}}x^{\frac{1}{2}}}
=\frac{1}{|\lambda-2|^{p-\frac{1}{2}}}\int_{0}^{\frac{2}{|\lambda-2|}}
\frac{du}{(u^2+1)^{\frac{p}{2}}u^{\frac{1}{2}}} \nonumber \\
&\leq& \frac{1}{|\lambda-2|^{p-\frac{1}{2}}}\int_{0}^{\infty}
\frac{du}{(u^2+1)^{\frac{p}{2}}u^{\frac{1}{2}}}
= \frac{C(p)}{|\lambda-2|^{p-\frac{1}{2}}} \nonumber \\
&=& \frac{C(p)|\lambda+2|^{\frac{1}{2}}}{\dist(\lambda,[-2,2])^{p-1}|\lambda^2-4|^{\frac{1}{2}}}
\leq \frac{C(p)}{\dist(\lambda,[-2,2])^{p-1}|\lambda^2-4|^{\frac{1}{2}}}.
\end{eqnarray}
The estimates (\ref{e0}) and (\ref{e1}) show the validity of (\ref{eq:gg}) in case that $\lambda_0 \geq 2$.
We now consider the case b):

b.1) If $\lambda_0 \in [0,2)$ and $\lambda_1>1$ then
\begin{eqnarray}\label{rrs}
&&\int_{0}^2\frac{dx}{|\lambda-x|^p(2-x)^{\frac{1}{2}}}\leq
\frac{1}{\dist(\lambda,[-2,2])^p}\int_{0}^2\frac{dx}{(2-x)^{\frac{1}{2}}} \nonumber \\
&=& \frac{C}{\dist(\lambda,[-2,2])^{p-1}|\lambda^2-4|^{\frac{1}{2}}}
\frac{|\lambda^2-4|^{\frac{1}{2}}}{\dist(\lambda,[-2,2])} \nonumber \\
&\leq&
\frac{C}{\dist(\lambda,[-2,2])^{p-1}|\lambda^2-4|^{\frac{1}{2}}}.
\end{eqnarray}
The last estimate shows the validity of (\ref{eq:gg}) in case that $\lambda_0 \in [0,2), \lambda_1 >1$.

b.2) Let $\lambda_0 \in [0,2)$ and $0< \lambda_1\leq 1$. Substituting
$u=\frac{\lambda_0-x}{\lambda_1}$, we obtain
\begin{eqnarray}
&&\int_{0}^2
\frac{dx}{|\lambda-x|^p(2-x)^{\frac{1}{2}}}=\int_{0}^2
\frac{dx}{((\lambda_0-x)^2+\lambda_1^2)^{\frac{p}{2}}(2-x)^{\frac{1}{2}}}\nonumber \\
&=&\frac{1}{\lambda_1^{p-1}}\int_{\frac{\lambda_0-2}{\lambda_1}}^{\frac{\lambda_0}{\lambda_1}}
\frac{du}{(u^2+1)^{\frac{p}{2}}(2-\lambda_0+\lambda_1
u)^{\frac{1}{2}}}\nonumber \\
&=&\frac{1}{\dist(\lambda,[-2,2])^{p-1}}\int_{\frac{\lambda_0-2}{\lambda_1}}^{\frac{\lambda_0}{\lambda_1}}
\frac{du}{(u^2+1)^{\frac{p}{2}}(2-\lambda_0+\lambda_1
u)^{\frac{1}{2}}},\nonumber\end{eqnarray} so in order to prove (\ref{eq:gg}) we
need to prove
\begin{equation}\label{w2}
\int_{\frac{\lambda_0-2}{\lambda_1}}^{\frac{\lambda_0}{\lambda_1}}
\frac{du}{(u^2+1)^{\frac{p}{2}}(2-\lambda_0+\lambda_1
u)^{\frac{1}{2}}}\leq \frac{C(p)}{|\lambda^2-4|^{\frac{1}{2}}}.
\end{equation}
Let us split the last integral into two parts. We have
\begin{eqnarray}&&
\int_{0}^{\frac{\lambda_0}{\lambda_1}}
\frac{du}{(u^2+1)^{\frac{p}{2}}(2-\lambda_0+\lambda_1
u)^{\frac{1}{2}}}\leq
\frac{1}{(2-\lambda_0)^{\frac{1}{2}}}\int_{0}^{\frac{\lambda_0}{\lambda_1}}
\frac{du}{(u^2+1)^{\frac{p}{2}}} \nonumber \\
&\leq& \frac{1}{(2-\lambda_0)^{\frac{1}{2}}}\int_{0}^{\infty}
\frac{du}{(u^2+1)^{\frac{p}{2}}}=\frac{C(p)}{(2-\lambda_0)^{\frac{1}{2}}}.\label{gg}
\end{eqnarray}
The integral in (\ref{gg}) can also be estimated in a different way
\begin{eqnarray}\label{gg1}&&
\int_{0}^{\frac{\lambda_0}{\lambda_1}}
\frac{du}{(u^2+1)^{\frac{p}{2}}(2-\lambda_0+\lambda_1
u)^{\frac{1}{2}}}\leq
\frac{1}{\lambda_1^{\frac{1}{2}}}\int_{0}^{\frac{\lambda_0}{\lambda_1}}
\frac{du}{(u^2+1)^{\frac{p}{2}}u^{\frac{1}{2}}}\nonumber\\
&\leq& \frac 1 {\lambda_1^{\frac 1 2}} \int_{0}^{\infty}
\frac{du}{(u^2+1)^{\frac{p}{2}}u^{\frac{1}{2}}}
=\frac{C(p)}{\lambda_1^{\frac{1}{2}}}.
\end{eqnarray}
Since we assumed that $\lambda_0 \in [0,2)$ and $\lambda_1 \in (0,1]$, the following holds:
\begin{equation*}
|\lambda^2-4|\leq
C|\lambda-2|=C[(2-\lambda_0)^2+\lambda_1^2]^{\frac{1}{2}} \leq C
\max(2-\lambda_0,\lambda_1),
\end{equation*}
implying that
\begin{equation}\label{daa}
\frac{1}{|\lambda^2-4|^{\frac{1}{2}}}
\geq C\min
\left(\frac{1}{(2-\lambda_0)^{\frac{1}{2}}},\frac{1}{\lambda_1^{\frac{1}{2}}}\right),
\end{equation}
and therefore (\ref{gg}) and (\ref{gg1}) show that
\begin{equation}\label{WW}\int_{0}^{\frac{\lambda_0}{\lambda_1}}
\frac{du}{(u^2+1)^{\frac{p}{2}}(2-\lambda_0+\lambda_1
u)^{\frac{1}{2}}}\leq \frac{C(p)}{|\lambda^2-4|^{\frac{1}{2}}}.
\end{equation}

To prove (\ref{w2}) it remains to prove
\begin{equation}\label{w3}
\int_{\frac{\lambda_0-2}{\lambda_1}}^{0}
\frac{du}{(u^2+1)^{\frac{p}{2}}(2-\lambda_0+\lambda_1
u)^{\frac{1}{2}}}\leq \frac{C(p)}{|\lambda^2-4|^{\frac{1}{2}}}.
\end{equation}

Making the change of variable $w=\frac{\lambda_1}{\lambda_0-2} u$,
we have
\begin{small}
\begin{eqnarray}\label{uu}
\int_{\frac{\lambda_0-2}{\lambda_1}}^{0}
\frac{du}{(u^2+1)^{\frac{p}{2}}(2-\lambda_0+\lambda_1
u)^{\frac{1}{2}}}=\frac{(2-\lambda_0)^{\frac{1}{2}}}{\lambda_1}
\int_0^{1}
\frac{dw}{((\frac{2-\lambda_0}{\lambda_1})^2w^2+1)^{\frac{p}{2}}(1-w)^{\frac{1}{2}}}.
\end{eqnarray}
\end{small}
To estimate the last integral we note that for any $c>0$,
\begin{eqnarray}\label{ur} && \int_0^{1}
\frac{dw}{(c^2w^2+1)^{\frac{p}{2}}(1-w)^{\frac{1}{2}}} \nonumber\\
&\leq &\sqrt{2}\int_0^{\frac{1}{2}}
\frac{dw}{(c^2w^2+1)^{\frac{p}{2}}}+\frac{2^p}{
(c^2+4)^{\frac{p}{2}}}\int_{\frac{1}{2}}^{1}
\frac{dw}{(1-w)^{\frac{1}{2}}}\nonumber\\
&=&\frac{\sqrt{2}}{c}\int_0^{\frac{c}{2}}
\frac{dv}{(v^2+1)^{\frac{p}{2}}}+\frac{C(p)}{(c^2+4)^\frac{p}{2}}
\leq \frac{\sqrt{2}}{c}\int_0^{\infty}
\frac{dv}{(v^2+1)^{\frac{p}{2}}}+\frac{C(p)}{(c^2+4)^\frac{p}{2}}\nonumber \\
&\leq& C(p)\left( \frac 1 c + \frac{1}{(c^2+4)^{\frac 1 2}} \right) \leq \frac{C(p)}{c}.
\end{eqnarray}

An alternative estimate leads to \begin{equation}\label{qr2}\int_0^{1}
\frac{dw}{(c^2w^2+1)^{\frac{p}{2}}(1-w)^{\frac{1}{2}}}\leq
\int_0^{1} \frac{dw}{(1-w)^{\frac{1}{2}}}=C.\end{equation}

Choosing $c= \frac{2-\lambda_0}{\lambda_1}$ we obtain from (\ref{uu}) and (\ref{ur}) that
\begin{eqnarray}\label{rew0}
\int_{\frac{\lambda_0-2}{\lambda_1}}^{0}
\frac{du}{(u^2+1)^{\frac{p}{2}}(2-\lambda_0+\lambda_1
u)^{\frac{1}{2}}}\leq \frac{C(p)}{(2-\lambda_0)^{\frac{1}{2}}}.
\end{eqnarray}
Similarly, from (\ref{uu}) and (\ref{qr2}) we have
\begin{eqnarray}\label{rew}
\int_{\frac{\lambda_0-2}{\lambda_1}}^{0}
\frac{du}{(u^2+1)^{\frac{p}{2}}(2-\lambda_0+\lambda_1
u)^{\frac{1}{2}}}\leq C(p)
\frac{(2-\lambda_0)^{\frac{1}{2}}}{\lambda_1}.
\end{eqnarray}
Distinguishing between the cases $\lambda_1\leq 2-\lambda_0$ and
$\lambda_1
> 2 - \lambda_0$, respectively, it is easy to check that
\begin{eqnarray*}
\min \left( \frac{1}{(2-\lambda_0)^{\frac 1 2 }}, \frac{(2-\lambda_0)^{\frac 1 2}}{\lambda_1} \right)
 \leq \min \left(\frac{1}{(2-\lambda_0)^{\frac{1}{2}}},\frac{1}{\lambda_1^{\frac{1}{2}}}\right),
\end{eqnarray*}
so (\ref{rew0}), (\ref{rew}) and (\ref{daa}) show the validity of (\ref{w3}). Finally, noting that (\ref{WW}) and (\ref{w3}) provide the proof of estimate (\ref{w2}), the proof of Lemma \ref{lem:resolv1} is completed.
\end{proof}
\section{Proof of Theorem \ref{thm4}}\label{final}

Let $p \geq \frac 3 2$ and $\tau \in (0,1)$. From (\ref{gk1}) we know that for $\theta \in [0, \frac \pi 2)$
\begin{eqnarray}\label{R0}
\sum_{\lambda\in \sigma_{d}(J) \cap \Omega_\theta^+} |\lambda-2|^{p-\frac 1 2} \leq  C(p)(1+2\tan(\theta))^{p} \|d\|_{l^{p}}^{p},
\end{eqnarray}
where $\Omega_\theta^{+}= \{ \lambda : 2 - \Re(\lambda) < \tan(\theta) |\Im \lambda| \}$. We define
\begin{equation*}
  \Psi_1 = \{ \lambda: \Re(\lambda) >0,\; 2- \Re(\lambda)< |\Im(\lambda)| \} \subset \Omega_{\pi /4}^+.
\end{equation*}
An easy calculation shows that for $\lambda \in \Psi_1$ we have
\begin{equation}
  |\lambda-2|^{p-\frac 1 2} \geq C(\tau) \frac{ \dist(\lambda,[-2,2])^{p +\tau}}{|\lambda^2-4|^{\frac 1 2 + \tau}},
\end{equation}
so (\ref{R0}) implies that
\begin{equation}\label{R11}
  \sum_{\lambda\in \sigma_{d}(J) \cap \Psi_1}  \frac{ \dist(\lambda,[-2,2])^{p +\tau}}{|\lambda^2-4|^{\frac 1 2 + \tau}} \leq C(p,\tau) \|d\|_{l^{p}}^{p}.
\end{equation}
Let $ \Psi_2 = \{ \lambda: \Re(\lambda)>0\} \setminus \Psi_1$ and set $x=\tan(\theta) \in [0,\infty)$. From (\ref{R0}) we obtain
\begin{equation}\label{R2}
  \sum_{\lambda\in \sigma_{d}(J)\cap \Psi_2, \: \frac{2-\Re(\lambda)  }{|\Im \lambda|} < x} |\lambda-2|^{p-\frac 1 2} \leq C(p)(1+2x)^{p} \|d\|_{l^{p}}^{p}.
\end{equation}
We multiply both sides of (\ref{R2}) with $x^{-p-1-\tau}$ and integrate with respect to $x \in [1,\infty)$. For the left-hand side we obtain
\begin{eqnarray*}
&&\int_1^\infty dx \; x^{-p-1-\tau}  \sum_{\lambda\in \sigma_{d}(J)\cap \Psi_2, \: \frac{2-\Re(\lambda)  }{|\Im \lambda|} < x} |\lambda-2|^{p-\frac 1 2} \\
&=& \sum_{\lambda \in \sigma_d(J) \cap \Psi_2} |\lambda-2|^{p-\frac 1 2} \int_{\max(1, \frac{2-\Re(\lambda)  }{|\Im \lambda|})}^\infty dx \; x^{-p-1-\tau} \\
&=& C(p,\tau) \sum_{\lambda \in \sigma_d(J) \cap \Psi_2} |\lambda-2|^{p-\frac 1 2} \left( \frac{|\Im \lambda|}{2-\Re(\lambda)  } \right)^{p+\tau} \\
&=& C(p,\tau) \sum_{\lambda \in \sigma_d(J) \cap \Psi_2} |\lambda-2|^{p-\frac 1 2} \left( \frac{\dist(\lambda,[-2,2])}{2-\Re(\lambda)  } \right)^{p+\tau} \\
&\geq& C(p,\tau) \sum_{\lambda \in \sigma_d(J) \cap \Psi_2}   \frac{\dist(\lambda,[-2,2])^{p+\tau}}{|\lambda^2-4|^{\frac 1 2 +\tau}  }.
\end{eqnarray*}
Similarly, for the right-hand side of (\ref{R2}) we obtain,
\begin{equation*}
 C(p) \|d\|_{l^{p}}^{p}  \int_1^\infty dx \; x^{-p-1-\tau} (1+2x)^{p} \leq C(p,\tau) \| d \|_{l^p}^p.
\end{equation*}
We have thus shown that
\begin{equation}\label{R3}
  \sum_{\lambda \in \sigma_d(J) \cap \Psi_2}   \frac{\dist(\lambda,[-2,2])^{p+\tau}}{|\lambda^2-4|^{\frac 1 2 +\tau}  } \leq C(p,\tau) \| d \|_{l^p}^p.
\end{equation}
Noting that $\Psi_1$ and $\Psi_2$ are disjoint with $\Psi_1 \cup
\Psi_2 = \{ \lambda : \Re(\lambda) > 0\}$  we conclude from
(\ref{R11}) and (\ref{R3}) that
\begin{equation*}
  \sum_{\lambda \in \sigma_d(J), \Re(\lambda) > 0}   \frac{\dist(\lambda,[-2,2])^{p+\tau}}{|\lambda^2-4|^{\frac 1 2 +\tau}  } \leq C(p,\tau) \| d \|_{l^p}^p.
\end{equation*}
Finally, starting with the estimate
\begin{eqnarray*}
\sum_{\lambda\in \sigma_{d}(J) \cap \Omega_\theta^-} |\lambda+2|^{p-\frac 1 2} \leq  C(p)(1+2\tan(\theta))^{p} \|d\|_{l^{p}}^{p},
\end{eqnarray*}
which follows from (\ref{gk1}), we can show in exactly the same manner as above that
\begin{equation*}
  \sum_{\lambda \in \sigma_d(J), \Re(\lambda) \leq 0}   \frac{\dist(\lambda,[-2,2])^{p+\tau}}{|\lambda^2-4|^{\frac 1 2 +\tau}  } \leq C(p,\tau) \| d \|_{l^p}^p.
\end{equation*}
This concludes the proof of Theorem \ref{thm4}.

\section*{Acknowledgement}
It's a pleasure to thank Michael Demuth for many valuable
discussions.

\end{document}